\documentclass[12pt, draftcls, onecolumn]{IEEEtran}
\IEEEoverridecommandlockouts
\usepackage{amsmath}
\usepackage{amssymb}
\usepackage{psfrag,graphicx}
\usepackage{epsfig}
\usepackage{verbatim}
\graphicspath{{../../Figures/}}
\newtheorem{thm}{Theorem}

\newtheorem{prp}{Proposition}
\newtheorem{defin}{Definition}

\newcommand{\sys}[4]{
\left\lgroup
\begin{array}{c|c}
  #1 & #2\\\hline
  #3 & #4
\end{array}
\right\rgroup
}

\def\q{{\bf{q}}}
\title{\LARGE \bf Optimal Control and Estimation of Partially Nested Interconnected Systems}
\author{Ather Gattami\thanks{
Ather Gattami is with the Automatic Control Laboratory, School of of Electrical Engineering,
KTH-Royal Institute of Technology, 100 44, Stockholm, Sweden. E-mail: gattami@kth.se.},
Sanjoy Mitter
\thanks{
Sanjoy Mitter is with the
Laboratory for Information and Decision Systems, Massachusetts
Institute of Technology, Cambridge, MA 02139, USA. E-mail:
mitter@mit.edu} }

\begin{document}
\maketitle
\begin{abstract}
In this paper, we study distributed estimation and control problems
over graphs under partially nested information patterns. We show a
duality result that is very similar to the classical duality
result between state estimation and state feedback control with a
classical information pattern, under the condition that the disturbances entering different systems on the graph are uncorrelated. The distributed estimation problem decomposes into $N$ separate estimation problems, where $N$ is the number of interconnected subsystems over the graph, and the solution to each subproblem is simply the optimal Kalman filter. This also gives the solution to the distributed control problem due to the duality of distributed estimation and control under partially nested information pattern. We then consider a weighted distributed estimation problem, where we get coupling between the estimators, and
separation between the estimators is not possible. We propose a solution based on linear quadratic team decision theory, which provides a generalized Riccati equation for teams. We show that the weighted estimation problem is the dual to the distributed state feedback problem, where the disturbances entering the interconnected systems are correlated. 
\end{abstract}

\begin{keywords}
Distributed Estimation and State Feedback Control, Duality.
\end{keywords}
\section{Introduction}
\subsection{Background}
Control with information structures imposed on the decision
maker(s) have been very challenging for decision theory
researchers. Even in the simple linear quadratic static decision
problem, it has been shown that complex nonlinear decisions could
outperform any given linear decision (see
\cite{witsenhausen:1968}). Important progress was made for the
stochastic static team decision problems in \cite{marschak:1955}
and \cite{radner}. New information structures were explored in
\cite{ho:chu} for the stochastic linear quadratic finite horizon
control problem. Similar algebraic conditions where given in
\cite{bamieh:voulgaris:2005} for homogeneous systems.
In \cite{rantzer:acc06}, the stationary state
feedback stochastic linear quadratic control problem was considered
using state space formulation, under the condition that all the
subsystems have a \textit{common past}, with the difficulty of recovering the structure
of the distributed controller. With common past, we mean
that all subsystems have information about the \textit{global}
state from some time step in the past. The time-varying
and stationary output feedback version was solved in
\cite{gattami:tac:10}. Recently, nice studies of Partial
Nestedness in linear quadratic dynamic team problems appeared in
Yuksel \cite{yuksel-pn-09} and Mahajan \textit{et al}
\cite{mahajan:tatikonda:09}. The $n$-step delay problem is studied in \cite{teneketzis_2011}.
Duality between estimation and control for distributed control problems of heterogeneous systems under arbitrary sparsity 
and delay partially nested structure, was explored \cite{gattami:acc09}, where state-feedback control and estimation was shown to be solved by a set of independent Riccati equations. In particular, \cite{gattami:acc09} showed that optimal controllers have a finite order 
for any partially nested information strucutre. A state-space solution for $N$ systems with no delays was given in \cite{shah} with a different approach relying on partially ordered set formulation. The work in \cite{vamsi} considers realizable solutions in the presence of noise. 

\subsection{Contribution}
In this paper, we will show a duality result between distributed
state estimation and distributed state feedback control under partially nested information including delays, similar to the centralized estimation and state feedback problems. Since the distributed control and estimation problems are dual, we show how to find the optimal distributed estimator (and hence the optimal distributed state-feedback controller). The distributed estimation problem decomposes into $N$ separate estimation problems, where $N$ is the number of interconnected subsystems of the network. We give an explicit solution for the three interconnected systems' case under two different graphs. The paper is an extension of  \cite{gattami:acc09} where we consider a more general framework. The general framework includes a weighted distributed estimation problem, where we get coupling between the estimators, and separation between the estimators is not possible. We propose a solution based on linear quadratic team decision theory, which provides a generalized Riccati equation for teams. We show that the weighted estimation problem is the dual to the distributed state feedback problem, where the disturbances entering the interconnected systems are correlated. The solutions do not assume stable systems, and a stabilizing solution is obtained automatically when it exists.


\subsection{Notation}
Let $\mathbb{R}$ be the set of real numbers, $\mathbb{Z}_2 = \{0,1\}$, 
$\mathbb{S}_{++}^n$ is the set of $n\times n$ positive definite matrices. $x\sim
\mathcal{N}(m,X)$ means that $x$ is a Gaussian variable with
$\mathbf{E}\{x\}=m$ and $\mathbf{E}\{(x-m)(x-m)^T\}=X$.
$[M]_i$, denotes the block row or column $i$ of a matrix $M$ 
depending on the context. For a matrix $A$ partitioned into blocks,
$[A]_{ij}$ denotes the block matrix of $A$ in block position $(i,j)$. 
 $I_{n}$ is the $n\times n$ identity matrix.
For vectors $v_k, v_{k-1}, ..., v_0$, we define $v_{[0,k]}:=\{v_k,
v_{k-1}, ..., v_0\}$.
We denote a discrete-time (stochastic) process $x(0), x(1), x(2),...$ by
$\{x(t)\}$. 
 The forward shift operator is denoted by $\q$, that is
$x(t+1)=\q x(t)$. A causal linear time-invariant operator $H(\q)$
maps a process $\{x(t)\}$ to an output $y(t)$, where $y(t)=H(\q^{-1})x(t)$,
and $H(\q^{-1})$ is given by its generating function (\cite{stanley}),
$H(\q^{-1})=\sum_{t=0}^{\infty} h(t)\q^{-t}$, $h(t)\in\mathbb{R}^{m
\times n}$. The norm of $\|H(\q^{-1})\|$ is defined as
$\|H(\q^{-1})\|^2 = \mathbf{E} \|H(\q^{-1})w(t)\|^2 = \sum_{t=0}^{\infty} \|h(t)\|^2 =  \sum_{t=0}^{\infty} \mathbf{Tr} [h^T(t)h(t)]$,
where $\{w(t)\}$ is a sequence of uncorrelated Gaussian variables with
$w(t)\sim \mathcal{N}(0,I)$. 
A transfer matrix in terms of state-space data is
denoted
$$\sys A B C D := C(\q I-A)^{-1}B+D.$$
\section{Linear Quadratic Team Theory}
 \label{teams}
Define a probability space $(\Omega, \mathcal{F}, \mathcal{P})$. 
Let $y_i$ be $p_i$-dimensional random variables, for $i = 1, ..., N$, and set
$p = p_1 + \cdots p_N$. Let $\mathcal{F}_i$ be the sigma field
generated by $y_i$.  
Introduce, $\mathcal{H}$, the space  of all $nN\times N$ matrices whose elements are
measurable functions from $\Omega$ to $\mathbb{R}$. 
Let $W\in \mathbb{S}_{++}^{N}$, and define 
\begin{equation}
\label{inner}
<H_1, H_2> = \mathbf{Tr}~\mathbf{E}\{H_1WH_2^T\}
\end{equation}
for $H_1, H_2\in \mathcal{H}$. Then,  $\mathcal{H}$ is a Hilbert space with inner product
(\ref{inner}) and norm $\|H\|_W^2=<H,H>$. Let $\mathcal{D}\subset \mathcal{H}$ such that
 for $D\in \mathcal{D}$, the $i$th column of $D$, $D_i$, is $\mathcal{F}_i$-measurable.
 
The columns of $D$, $D_1, ..., D_N$, make up a team, where the \emph{players} $D_i$ make decisions in local information given by $y_i$, to minimize a cost of the form

$$
<D-\Phi X,D-\Phi X>
$$
for some $\Phi\in \mathbb{R}^{nN\times nN}$ and $X\in \mathcal{H}$. \\

 \begin{prp}
 \label{orth}
Let $X\in \mathcal{H}$. The minimum of $<D-X,D-X>$ for $D\in \mathcal{D}$
is acheived by the unique $\hat{X}\in \mathcal{D}$ satisfying
$$
<\hat{X}-X,D> = 0
$$
for all $D\in \mathcal{D}$.
 \end{prp}
 \begin{proof}
 Consult \cite{barta:sandell}.
 \end{proof}

The following proposition gives a certainty equivalence property for
team problems:
 \begin{prp}
 \label{sep}
Let 
$$\hat{X} = \arg \min_{D\in \mathcal{D}} <D-X,D-X>.$$
Then, 
$$
\Phi \hat{X}= \arg \min_{D\in \mathcal{D}}<D-\Phi X,D-\Phi X>. 
$$
 \end{prp}
 \begin{proof}
 Consult \cite{barta:sandell}.
 \end{proof}
 
\begin{prp}
\label{equi}
Let $u_i$ be $n$-dimensional vectors and  $L_i\in\mathbb{R}^{n\times n}$, for $i = 1, ..., N$, and $\Phi\in \mathbb{R}^{nN\times nN}$ with
$$
\Phi = 
\begin{bmatrix}
L_1 ~ \cdots ~ L_N\\
0
\end{bmatrix}.
$$
Let
$$
u^* = \arg \min_{u_i\in \mathcal{F}_i} \mathbf{E}\{(u-Lx)^TW(u-Lx)\}
$$
and 
$$
\hat{X}= \arg \min_{D\in \mathcal{D}}<D - X,D - X> 
$$
Then, 
$$
u_i^* = \sum_{j=1}^N = L_j \hat{X}_{ji}.
$$
 \end{prp}
\begin{proof}
Consult \cite{barta:sandell}.\\
\end{proof}
\begin{defin}
$X$ is $W$-orthogonal to $Y$ if
$<X, Y> = 0$.
\end{defin}
\begin{defin}
For a sequence $\{Z_k\}_{Z_k\in \mathcal{D}}$,
$\{Z_k\}$ is called $W$-white noise if $<Z_k, Z_{l}> = 0$ for all $k\neq l$.
\end{defin}

Note that for $W=I$, we get the formal definition of white noise in the classical sense.

Now introduce the matrix
$$
Y = \mathbf{diag}( y_1, \cdots,  y_N).
$$

The next proposition shows how to obtain the linear optimal solution $D = KY$:
\begin{prp}
  \label{lin}
Let $X\in \mathcal{H}$. The minimum of $\|KY-X\|_W$  over $K\in \mathbb{R}^{nN\times p}$ 
with $KY\in \mathcal{D}$ is acheived by the unique $K^\star$  given by
$$
K^\star =  \mathbf{E} \{XWY^T\}  (\mathbf{E}  \{YWY^T\})^{-1}.
$$
 \end{prp}
 \begin{proof}
Consult \cite{barta:sandell}.
 \end{proof}


\section{Systems over Graphs}
\label{systems_over_graphs}
Consider linear systems $P_i(\q^{-1})$
with state space realization
\begin{equation}
\label{IS}
\begin{array}{ll}
\begin{aligned}
x_i(t+1) &=\sum_{j=1}^{N} A_{ij}x_{j}(t)+B_iu_i(t)+w_i(t)\\
y_i(t)   &=C_ix_i(t) + v_i(t),
\end{aligned}
\end{array}
\end{equation}
for $i=1,...,N$. Here, $A_{ij}\in \mathbb{R}^{n_i\times n_j}$,
$B_{i}\in \mathbb{R}^{n_i\times m_i}$, and $C_i\in \mathbb{R}^{p_i\times n_i}$.
$w_i$ is the disturbance and $u_i$ is the control signal, entering system $i$. Also, we have
that $\sum_i m_i = m$,  $\sum_i n_i = n$,  $\sum_i p_i = p$.

The systems are interconnected as follows. If the state of system
$j$ at time step $t$ (i.e., $x_j(t)$) affects the state of system $i$ at time step $t+1$ (i.e., $x_i(t+1)$),
then $A_{ij}\neq 0$, otherwise $A_{ij}=0$. 
This block structure can be described by a graph\footnote{See the Appendix for
a short introduction to graph theory.} $\mathcal{G}$ of
order $N$, whose adjacency matrix is $\mathcal{A}$. The graph
$\mathcal{G}$ has an arrow from node $j$ to $i$ if and only if
$A_{ij}\neq 0$. The transfer function of the interconnected
systems is given by $P(\q^{-1})=C(\q I-A)^{-1}B$. Then, the system
$P^T(\q^{-1})$ is equal to $B^T(\q I-A^T)^{-1}C^T$, and it can be
represented by a graph $\mathcal{G^*}$ which is the adjoint of
$\mathcal{G}$, since the adjacency matrix of $\mathcal{G^*}$ is
$\mathcal{A}^*=\mathcal{A}^T$. The block diagram for the
transposed interconnection is simply obtained by reversing the
orientation of the interconnection arrows. This property was
observed in \cite{bernhardsson:1992}.

For any generating function $F(\lambda)$, we write the
generating function
\begin{equation*}
\begin{aligned}
G(\lambda)  &=(I-F(\lambda))^{-1}=\sum_{t\geq 0} (F(\lambda))^t.
\end{aligned}
\end{equation*}

\begin{defin}[Sparsity Structure] Let $m, n, N$ be integers with $m,n\geq N$, 
$\mathcal{A}\in \mathbb{Z}_2^{N\times N}$, and

$$S_{\mathcal{A}}^{m\times n}=\left\{\sum_{t\geq 0}g(t)\lambda^t \Big{|}
g(t)\in \mathbb{R}^{m\times n}, [\mathcal{A}^t]_{ij} = 0 \Rightarrow [g(t)]_{ij}=0 \right \}.$$ 

We say that $G(\lambda)$ has the sparsity structure given by $\mathcal{A}$ if $G(\lambda)\in S_{\mathcal{A}}$.\\
\end{defin}

\begin{thm}
\label{mult_generating_f}
Suppose that $G_1(\lambda)\in S_{\mathcal{A}}^{m\times p}$, $G_2(\lambda)\in S_{\mathcal{A}}^{p\times n}$ for a given adjacency matrix 
$\mathcal{A}\in \mathbb{Z}_2^{N\times N}$. Then $G_1(\lambda)G_2(\lambda) \in S_{\mathcal{A}}^{m\times n}$.
\end{thm}
\begin{proof}
Let $G_1(\lambda)=\sum_{t\geq 0}g_1(t)\lambda^t$ and $G_2(\lambda)=\sum_{t\geq 0}g_2(t)\lambda^t$. 
Then, $G_3(\lambda) = G_1(\lambda)G_2(\lambda)=\sum_{t\geq 0}g_3(t)\lambda^t$, where 
$g_3(t)=\sum_{s= 0}^t  g_1(s)g_2(t-s)$. Let $r$ denote the $i$:th row of 
$\mathcal{A}^s$ and $c$ denote the $j$:th column of $\mathcal{A}^{t-s}$. Then 
$\mathcal{A}^t_{ij} = [\mathcal{A}^s\cdot \mathcal{A}^{t-s}]_{ij}= r\cdot  c$. Now $\mathcal{A}^t_{ij} = 0$, implies 
that $r\cdot c= 0$. Since $r$ and $c$ consist of non-negative integers, we have either $r_k=0$ or $c_k=0$, for all $k$. 
In an analog manner, let  $u$ denote the $i$:th block row of $g_1(s)$ and $v$ the $j$:th block column of
$g_2(t-s)$. Clearly, $r_j = [\mathcal{A}^s]_{ij} = 0$ implies that $u_j=0$, and $c_i=[\mathcal{A}^{(t-s)}]_{ij} = 0$ implies 
that $v_i=0$. Thus, for all $k$, either $u_k$ or $v_k$ is zero, that is $u_kv_k=0$. Hence,
$[g_1(s)g_2(t-s)]_{ij}=u\cdot v = u_1v_1+ \cdots + u_Nv_N =0$, and so $[g_3(t)]_{ij}=0$. We conclude that
$[\mathcal{A}^t]_{ij} = 0 \Rightarrow [g_3(t)]_{ij}=0$, and so 
$G_3(\lambda) = \sum_{t\geq 0}g_3(t)\lambda^t\in S_{\mathcal{A}}^{m\times n}$.\\
\end{proof}

\begin{thm}
\label{invariance}
Let $\mathcal{A}$ be a given adjacency matrix,
$H_1(\lambda) \in \mathcal{S}_{\mathcal{A}}^{n\times n}$, and 
$H_2(\lambda)\in \mathcal{S}_{\mathcal{A}}^{m\times n}$. Then
$H_2(\lambda)(I-H_1(\lambda))^{-1}\in \mathcal{S}_{\mathcal{A}}^{m\times n}.$
\end{thm}
\begin{proof} Let $H_3(\lambda)=H_2(\lambda)(I-H_1(\lambda))^{-1}\in \mathcal{S}_{\mathcal{A}}^{m\times n}.$
The formal power series of $H_3(\lambda)$ is
$$
H_3(\lambda)=\sum_{s\geq 0} H_2(\lambda)(H_1(\lambda)))^{s}.
$$ 
Recursive use of Theorem \ref{mult_generating_f} implies that 
$H_2(\lambda)(H_1(\lambda)))^{s}\in S_\mathcal{A}^{m\times n}$ for all $s\geq 0$.
Hence, $H_3(\lambda)\in S_\mathcal{A}^{m\times n}$, and the proof is complete.\\
\end{proof}

{\bf Remark.} Theorem \ref{mult_generating_f} gives a more general invariance property than
Quadratic Invariance \cite{rotkowitz:lall:2006} in our case. We show that the structure of $H_1(\lambda)$ of 
$H_2(\lambda)$ is preserved under multiplication, and Theorem  \ref{invariance} shows that $H_1(\lambda)$ under negative feedback of $H_2(\lambda)$, the structure of the closed loop $H_2(\lambda)(I-H_1(\lambda))^{-1}$ is preserved. For Quadratic Invariance, $K(\lambda)(I-G(\lambda)K(\lambda))^{-1}$ has the same structure as $G(\lambda)$ and $K(\lambda)$ if and only if $K(\lambda)G(\lambda)K(\lambda)$ has the same structure as $K(\lambda)$ and $G(\lambda)$. In our case, taking $H_2(\lambda) = K(\lambda)$ and $H_1(\lambda) = G(\lambda)K(\lambda)$, then if $G(\lambda)$ and $K(\lambda)$ have the
same structure, then so does $G(\lambda)K(\lambda)=H_1(\lambda)$. It implies that both $H_2(\lambda)H_1(\lambda)=K(\lambda)G(\lambda)K(\lambda)$ and $H_2(\lambda)(I-H_1(\lambda))^{-1}$ have the the same structure as $G(\lambda)$ and $K(\lambda)$.

%
%


\section{Duality of Estimation and Control}
\subsection{Distributed State Feedback Control}
Consider the interconnected systems
\begin{equation}
\label{sfbIS}
\begin{array}{ll}
\begin{aligned}
x_i(t+1) &=\sum_{j=1}^{N} A_{ij}x_{j}(t)+B_{ii}u_i(t)+w_i(t)\\
y_i(t)   &=x_i(t),
\end{aligned}
\end{array}
\end{equation}
$w(t)\sim \mathcal{N}(0,I)$ for all $t$, and $x(t)=0$ for all
$t\leq 0$. Without loss of generality, we assume that $B_{ii}$ has
full column rank, for $i=1, ..., N$ (and hence has a left
inverse). 

The problem we are considering here is to find the optimal
distributed state feedback control
\begin{equation}
\label{lc} u_i(t) = K_i(\q^{-1})x(t) = \sum_{s=0}^{\infty} k_i(s) x(t-s),
\end{equation}
for $i=1, ..., N$ that minimizes the quadratic cost
$$J(x,u):=\lim_{M\rightarrow \infty}\frac{1}{M}\sum_{t=1}^M \mathbf{E} \|Cx(t)+Du(t)\|^2,$$
The partially nested information pattern is reflected in the parameters $k_i(s)$,
where $k_{ij}(s)=0$ if $[\mathcal{A}^s]_{ij}=0$, and
$\mathcal{A}\in \mathbb{Z}_2^{N\times N}$ is the adjacency matrix
of the interconnection graph . Thus, the block sparsity structure
of $K(\q^{-1})$ is the same as the sparsity structure of
$$(I-\q^{-1}\mathcal{A})^{-1}=I +\mathcal{A}\q^{-1}+\mathcal{A}^2\q^{-2}+\cdots ,$$ 
and so $K(\lambda)\in S_{\mathcal{A}}^{m\times n}$. 
To summarize, the problem we are considering is:
\begin{equation}
\label{controlP}
    \begin{aligned}
        \inf_{K(\lambda)\in S_{\mathcal{A}}^{m\times n}}\hspace{2mm} & \lim_{M\rightarrow \infty} \frac{1}{M}\sum_{t=1}^M \sum_{i=1}^N \mathbf{E} \|z_i(t)\|^2 \\
        \text{subject to }\hspace{2mm} & x(t+1) = Ax(t)+Bu(t)+w(t)\\
        & z(t)=Cx(t)+Du(t)\\
        & B= \text{diag}(B_{11}, ..., B_{NN})\\
        & u(t)=\sum_{s=0}^{\infty} k(s) x(t-s)\\
        & w(t)=x(t)=x(0)=0\hspace{2mm} \text{for all } t < 0\\
        & w(t)\sim \mathcal{N}(0,I)\hspace{2mm} \text{for all } t\geq 0 
    \end{aligned}
\end{equation}

\subsection{Distributed Feedforward Control}
The feedforward control problem is closely related to the state-feedback problem:
\begin{equation}
\label{feedforward}
    \begin{aligned}
        \inf_{G(\lambda)\in S_\mathcal{A}^{m\times n}}\hspace{2mm} & \lim_{M\rightarrow \infty} \frac{1}{M}\sum_{t=1}^M \sum_{i=1}^N \mathbf{E} \|z_i(t)\|^2 \\
        \text{subject to }\hspace{2mm} & x(t+1) = Ax(t)+Bu(t)+w(t)\\
        & B= \text{diag}(B_{11}, ..., B_{NN})\\
        & z(t)=Cx(t)+Du(t)\\
        & u(t)=-\sum_{s=0}^{\infty} g(s)w(t-1-s)\\
        & w(t)=x(t)=x(0)=0\hspace{2mm} \text{for all } t < 0\\
        & w(t)\sim \mathcal{N}(0,I)\hspace{2mm} \text{for all } t\geq 0
   \end{aligned}
\end{equation}
Note that (\ref{controlP}) and (\ref{feedforward}) are not equivalent in general, since
the latter only uses information about that external signals, $w$, entering the system, whereas 
for more restrictive informations structures, the control signals could carry information 
(see \cite{witsenhausen:1968} and \cite{ho:chu}).

\subsection{Distributed State Estimation}
Consider $N$ systems given by
\begin{equation}
\label{estIS}
\begin{array}{ll}
\begin{aligned}
x(t+1) &=Ax(t)+Bw(t)\\
y_i(t) &=C_{ii}x_i(t)+D_iw(t),
\end{aligned}
\end{array}
\end{equation}
for $i=1, ..., N$, $w(t)\sim \mathcal{N}(0,I)$, and $x(t)=0$ for
all $t\leq 0$. Without loss of generality, we assume that $C_i$
has full row rank, for $i=1, ..., N$. The problem is to find
optimal distributed estimators $\hat{x}_i(t)$  to minimize the cost
\begin{equation}
\label{decEst} \lim_{M\rightarrow \infty}
\frac{1}{M}\sum_{t=1}^M \sum_{i=1}^N \mathbf{E} \|x_i(t)-\hat{x}_i(t)\|^2
\end{equation}
In a similar way to the distributed state feedback problem, the
information pattern is the partially nested, which is reflected by
the interconnection graph, so $L(\lambda)\in S_{\mathcal{A}}^{n\times m}$. The linear decisions are optimal, hence
we can assume that
\begin{equation}
\label{le}
    \begin{aligned}
        \hat{x}_i(t) &= L({\q^{-1}})y(t-1) = \sum_{s=0}^{\infty} l_i(s)y(t-1-s).
    \end{aligned}
\end{equation}
Then, our problem becomes
\begin{equation}
\label{estimationP}
    \begin{aligned}
        \inf_{L(\lambda)\in S_{\mathcal{A}}^{n\times m}}\hspace{2mm} & \lim_{M\rightarrow \infty} \frac{1}{M}\sum_{t=1}^M \sum_{i=1}^N \mathbf{E} \|x_i(t)-\hat{x}_i(t)\|^2 \\
        \text{subject to }\hspace{2mm}
        & x(t+1) = Ax(t)+Bw(t)\\
        & y(t)=Cx(t)+Dw(t)\\
        & C= \text{diag}(C_{11}, ..., C_{NN})\\
        &\hat{x}(t)=\sum_{s=0}^{\infty}l(s) y(t-1-s)\\
        & w(t)=x(t)=x(0)=0\hspace{2mm} \text{for all } t < 0\\
        & w(t)\sim \mathcal{N}(0,I)\hspace{2mm} \text{for all } t\geq 0
    \end{aligned}
\end{equation}

In the next section, we will show the connection between the three problems that were introduced in this section.
\section{Duality Results}

\subsection{Duality of State-Feedback and Feedforward Control}

\begin{thm}
\label{equivalencethm}
The problems  (\ref{controlP}) and (\ref{feedforward})  are in bijection.
\end{thm}
\begin{proof} First write
\begin{equation}
    \begin{aligned}
        x(t) &=(\q I-A-BK(\q^{-1}))^{-1}w(t)\\
            &=(I-A\q^{-1}-BK(\q^{-1})\q^{-1})^{-1}\q^{-1}w(t).
    \end{aligned}
\end{equation}
Then
 $$u(t)= -K(\q^{-1})(I-A\q^{-1}-BK(\q^{-1})\q^{-1})^{-1} w(t-1).$$
Set $H_1(\lambda) = A\lambda+BK(\lambda)\lambda$ and $H_2(\lambda)=-K(\lambda)$.
Since $B\lambda \in \mathcal{S}_{\mathcal{A}}^{n\times m}$, Theorem \ref{mult_generating_f} implies that $B\lambda K(\lambda)\in \mathcal{S}_{\mathcal{A}}^{n\times n}$, and thus 
$H_1(\lambda)\in \mathcal{S}_{\mathcal{A}}^{n\times n}$. Now applying Theorem
\ref{invariance}, we get 
$$
G(\lambda) = -K(\lambda)(I-A\lambda-BK(\lambda)\lambda)^{-1}\in \mathcal{S}_{\mathcal{A}}^{m\times n}.
$$
In a similar way, we find that
$$
K(\q^{-1}) = G(\q^{-1})(I-A\q^{-1}-BK(\q^{-1})\q^{-1})
$$
$$\Updownarrow$$
$$
K(\q^{-1}) + G(\q^{-1})BK(\q^{-1})\q^{-1}= G(\q^{-1})(I-A\q^{-1})
$$
$$\Updownarrow$$
\begin{equation}
\begin{aligned}
K(\q^{-1}) 	&= (I+G(\q^{-1})B\q^{-1})^{-1}G(\q^{-1})(I-A\q^{-1})\\
		&= G(\q^{-1})(I+B\q^{-1}G(\q^{-1}))^{-1}(I-A\q^{-1}).
\end{aligned}
\end{equation}
Applying theorems \ref{mult_generating_f} and \ref{invariance} to the generating function above shows that
$G(\lambda)\in S_\mathcal{A}^{m\times n} \Rightarrow K(\lambda)\in S_\mathcal{A}^{m\times n}$. Hence, there is a bijection between the two controllers $K$ and $G$, and the proof is complete.
\end{proof}

\subsection{Duality of Distributed Estimation and Feedforward Control}
\begin{thm}
\label{estimationthm}
Consider the distributed feedforward linear quadratic problem
(\ref{controlP}), with state space realization
$$
\sys {A} {\begin{matrix} I & B\end{matrix}} {\begin{matrix} C \\
0\end{matrix}} {\begin{matrix} 0 & D\\ -I & 0\end{matrix}}
$$
and solution $u(t)=-\sum_{s=0}^\infty g(s)w(t-s)$, $\sum_{t=0}^{\infty} g(t)\lambda^{t} \in S_\mathcal{A}^{m\times n}$, and the
distributed estimation problem (\ref{estimationP}) with state
space realization
$$
\sys {A^T} {\begin{matrix} C^T & 0\end{matrix}} {\begin{matrix} I
\\ B^T\end{matrix}} {\begin{matrix} 0 & -I\\ D^T & 0\end{matrix}}
$$
and solution $\hat{x}(t)=\sum_{s=0}^\infty l(s)y(t-s-1)$, $\sum_{t=0}^{\infty} l(t)\lambda^{t} \in S_{\mathcal{A}^T}^{n\times m}$.
Then, for all $s$, $g(s)=l^T(s)$.
\end{thm}
\begin{proof} 
Introduce an uncorrelated Gaussian process $\bar{w}(t)\sim \mathcal{N}(0,I)$ with proper dimensions. 
For any transfer function $F$, we have that $\mathbf{E}\|F(\q^{-1})\bar{w}(t)\|^2 = \|F(\q^{-1})\|^2$. Using this fact we see that
each term in the quadratic cost of (\ref{controlP}) can be written as
\begin{equation}
\label{quadraticterm}
    \begin{aligned}
        \mathbf{E}\|  Cx(t)+Du(t)\|^2 &= \mathbf{E}\left\| C(\q I-A)^{-1}w(t)-[C(\q I-A)^{-1}B+D]G(\q^{-1})\q^{-1}w(t)\right\|^2\\
        &= \left\|  C(\q I-A)^{-1} - [C(\q I-A)^{-1}B+D]G(\q^{-1})\q^{-1}\right\|^2\\
        &= \left\|  (\q I-A^T)^{-1}C^T - G^T(\q^{-1})\q^{-1}[B^T(\q I-A^T)^{-1}C^T+D^T]\right\|^2\\
        &= \mathbf{E} \left\|  (\q I-A^T)^{-1}C^T\bar{w}(t) - 
        								G^T(\q^{-1})\q^{-1}[B^T(\q I-A^T)^{-1}C^T+D^T]\bar{w}(t)\right\|^2,
    \end{aligned}
\end{equation}
where the third equality is obtained from transposing which
doesn't change the value of the norm. 
Introduce the state space equation
\begin{equation}
    \begin{aligned}
         \bar{x}(t+1) &= A^T\bar{x}(t)+C^T\bar{w}(t)\\
         y(t)   &= B^T\bar{x}(t)+D^T\bar{w}(t)\\
    \end{aligned}
\end{equation}
and let $$\hat{x}(t)=G^T(\q^{-1})y(t-1).$$ Then comparing with
(\ref{quadraticterm}), we see that
\begin{equation}
    \begin{aligned}
        \mathbf{E}\|Cx(t)+Du(t)\|^2 &= \mathbf{E}\|\bar{x}(t)-\hat{x}(t)\|^2
                                    = \sum_{i=1}^{N}\mathbf{E}\|\bar{x}_i(t)-\hat{x}_i(t)\|^2.
    \end{aligned}
\end{equation}
The solution of the control problem described as a feedforward
problem, $G(\q^{-1})\in S_\mathcal{A}^{m\times n}$, is equal to $L^T(\q^{-1})\in S_{\mathcal{A}^T}^{n\times m}$, 
where $L(\q^{-1})$ is the solution of the corresponding dual estimation problem. 
\end{proof}

We have transformed the feedforward control problem to an estimation
problem, where the parameters of the estimation problem are the
transposed parameters of the control problem:
\begin{equation}
    \begin{aligned}
        &A \leftrightarrow A^T\\
        &B \leftrightarrow C^T\\
        &C \leftrightarrow B^T\\
        &D \leftrightarrow D^T
    \end{aligned}
\end{equation}

Note that we can have a distributed estimation problem with 
controller of the form $u(t) = K(\q^{-1})y(t-1) $ with  $K(\lambda)\in S_{\mathcal{A}}^{p\times m}$:
\begin{equation}
\label{estimationPC}
    \begin{aligned}
        \inf_{L(\lambda)\in S_{\mathcal{A}}^{n\times m}}\hspace{2mm} & \lim_{M\rightarrow \infty} \frac{1}{M}\sum_{t=1}^M \sum_{i=1}^N \mathbf{E} \|x_i(t)-\hat{x}_i(t)\|^2 \\
        \text{subject to }\hspace{2mm}
        & x(t+1) = Ax(t)+B_1w(t)+B_2u(t) \\
        & y(t)=Cx(t)+Dw(t)\\
        & B_2 = \mathbf{diag}(B_{11}, ..., B_{NN})\\
        & C= \mathbf{diag}(C_{11}, ..., C_{NN})\\
        &\hat{x}(t)=\sum_{s=0}^{\infty}l(s) y(t-1-s)\\
        & w(t)=x(t)=x(0)=0\hspace{2mm} \text{for all } t < 0\\
        & w(t)\sim \mathcal{N}(0,I)\hspace{2mm} \text{for all } t\geq 0
    \end{aligned}
\end{equation}
By considering the controller $u(t)$ as a propagating mean, the problem (\ref{estimationPC}) 
is essentially the same as (\ref{estimationP}) (compare with the centralized Kalman Filter).
\section{The Optimal Controller and Estimator}
\label{contribution}
Since the distributed control and estimation problems are dual, we will show how to find the optimal distributed estimator (and hence the optimal distributed state-feedback controller by just transposing the optimal distributed estimator). In particular, we will present two examples of three interconnected systems with both \textit{sparsity and delayed} measurements. First we consider an acyclic graph and then a connected graph.  Connected graphs possess a property of common information that is absent in acyclic graphs. Naturally, any graph can be written as clusters of connected graphs, interconnected over an acyclic graph, and these can be put together using our framework.

\subsection{Optimal Distributed Estimators}
Consider the estimation problem given by (\ref{estimationP}) (problem  (\ref{estimationPC}) can be treated similarly). 
It can be decomposed into $N$ \textit{decoupled} and \textit{centralized} estimation problems according to
\begin{equation}
\label{distr_estimationP}
    \begin{aligned}
        \inf_{l_i(s), s\geq 0}\hspace{2mm} & \lim_{M\rightarrow \infty} \frac{1}{M}\sum_{t=1}^M \mathbf{E} \|x_i(t)-\hat{x}_i(t)\|^2 \\
        \text{subject to }\hspace{2mm}
        & x(t+1) = Ax(t)+Bw(t)\\
        & y(t)=Cx(t)+Dw(t)\\                
        & C= \text{diag}(C_{11}, ..., C_{NN})\\
        &\hat{x}_i(t)=\sum_{s=0}^{\infty}l_i(s) y(t-1-s)\\
        & l_{ij}(s) = 0 \hspace{3mm} \text{if } [\mathcal{A}^s]_{ij}=0, s\geq 0\\
        & w(t)=x(t)=x(0)=0\hspace{2mm} \text{for all } t < 0\\
        & w(t)\sim \mathcal{N}(0,I)\hspace{2mm} \text{for all } t\geq 0
    \end{aligned}
\end{equation}
for $i=1, ..., N$. By introducing the augmented vector of delayed measurements $Y(t-1)=(y(t-1), y(t-2), ..., y(t-N))$, the optimal solution is the optimal Kalman filter with respect to a \textit{subset of blocks} of the augmented vector $Y(t-1)$, which is defined by the structure of $l_i(s), s\geq 0$.


We will illustrate how to obtain a state-space solution to the optimal distributed filtering problem 
for the case of three interconnected systems over two different graphs. By duality (Theorems \ref{equivalencethm} and  \ref{estimationthm}), it is equivalent to finding the state-space solution for the distributed optimal control problem. 
The interconnection is  defined by the system matrix 
$$
A=\left[
		\begin{matrix}
			A_{11} 	& A_{12} 	& A_{13}\\
			A_{21}	& A_{22} 	& A_{23}\\
			A_{31}	& A_{32}	& A_{33}
		\end{matrix}
		\right].
$$
First consider three interconnected systems over a chain, given by the state-space realization
\begin{equation}
	\begin{aligned}
		\left[
		\begin{matrix}
			x_1(t+1)\\
			x_2(t+1)\\
			x_3(t+1)
		\end{matrix}
		\right] &= \left[
		\begin{matrix}
			A_{11} 	& A_{12} 	& 0\\
			0		& A_{22} 	& A_{23}\\
			0		& 0			& A_{33}
		\end{matrix}
		\right] \left[
		\begin{matrix}
			x_1(t)\\
			x_2(t)\\
			x_3(t)
		\end{matrix}
		\right] +\left[
		\begin{matrix}
			B_1\\
			B_2\\
			B_3
		\end{matrix}
		\right] w(t) \\
		 \left[
		\begin{matrix}
			y_1(t)\\
			y_2(t)\\
			y_3(t)
		\end{matrix}
		\right] & = \left[
		\begin{matrix}
			C_{11} 	& 0		& 0\\
			0		& C_{22} 	& 0\\
			0		& 0		& C_{33}
		\end{matrix}
		\right] \left[
		\begin{matrix}
			x_1(t)\\
			x_2(t)\\
			x_3(t)
		\end{matrix}
		\right] +\left[
		\begin{matrix}
			D_1\\
			D_2\\
			D_3
		\end{matrix}
		\right] w(t).
	\end{aligned}
\end{equation}
The adjacency matrix of the interconnection is graph is given by $\mathcal{A}$, and we have
$$
\mathcal{A}^0= \left[
		\begin{matrix}
			1 	& 0 	& 0\\
			0	& 1 	& 0\\
			0	& 0	& 1
		\end{matrix}
		\right], \hspace{3mm}
\mathcal{A}= \left[
		\begin{matrix}
			1 	& 1 	& 0\\
			0	& 1 	& 1\\
			0	& 0	& 1
		\end{matrix}
		\right], \hspace{3mm} 
	\mathcal{A}^2 = \left[
		\begin{matrix}
			1 	& 2 	& 1\\
			0	& 1 	& 2\\
			0	& 0	& 1
		\end{matrix}
		\right], \hspace{3mm} 
		\mathcal{A}^s = \left[
		\begin{matrix}
			* 	& * 	& *\\
			0	& * 	& *\\
			0	& 0	& *
		\end{matrix}
		\right]\hspace{3mm}  \forall s>2,
$$ 
where the stars stand for positive integers.		
The condition $ [\mathcal{A}^s]_{ij}=0\Rightarrow l_{ij}(s)=0$ 
implies that the information available to estimate $\hat{x}_1(t)$ is given by 		
$y_1(t-1-s)$ for all $s\geq 0$ (since $[\mathcal{A}^s]_{11}=1, \forall s\geq 0$), $y_2(t-1-s)$ for all $s\geq 1$  
(since $[\mathcal{A}^s]_{12}=1, \forall s\geq 1$), and $y_3(t-1-s)$ for all $s\geq 2$  (since $[\mathcal{A}^s]_{13}=1, \forall s\geq 2$). The problem of estimating $x_1(t)$ based on information induced by the sparisity structure of $\mathcal{A}$ can be written as 
a centralized estimation problem, by an algebraic lifting, with respect to the \textit{extended} system dynamics

\begin{equation}
\label{extended}
	\begin{aligned}
		x_e(t)=\left[
		\begin{matrix}
			x_1(t)\\
			x_2(t)\\
			x_3(t)\\
			y_1(t-1)\\
			y_2(t-1)\\
			y_3(t-1)\\
			y_1(t-2)\\
			y_2(t-2)\\
			y_3(t-2)
		\end{matrix}
		\right] &= \left[
		\begin{matrix}
			A_{11} 	& A_{12} 	& 0 			& 0 	& 0 	& 0 	& 0 	& 0 	& 0\\
			0		& A_{22} 	& A_{23}		& 0 	& 0 	& 0 	& 0 	& 0 	& 0\\
			0		& 0			& A_{33}		& 0 	& 0 	& 0 	& 0 	& 0 	& 0\\
			C_{11} 	& 0		 	& 0 			& 0 	& 0 	& 0 	& 0 	& 0 	& 0\\
			0		& C_{22} 		& 0			& 0 	& 0 	& 0 	& 0 	& 0 	& 0\\
			0		& 0			& C_{33}		& 0 	& 0 	& 0 	& 0 	& 0 	& 0\\
			0	 	& 0		 	& 0 			& I 	& 0 	& 0 	& 0 	& 0 	& 0\\
			0		& 0	 		& 0			& 0 	& I 	& 0 	& 0 	& 0 	& 0\\
			0		& 0			& 0			& 0 	& 0 	& I 	& 0 	& 0 	& 0
		\end{matrix}
		\right] \left[
		\begin{matrix}
			x_1(t-1)\\
			x_2(t-1)\\
			x_3(t-1)\\
			y_1(t-2)\\
			y_2(t-2)\\
			y_3(t-2)\\
			y_1(t-3)\\
			y_2(t-3)\\
			y_3(t-3)
		\end{matrix}
		\right] +\left[
		\begin{matrix}
			B_1\\
			B_2\\
			B_3\\
			D_1\\
			D_2\\
			D_3\\
			0\\
			0\\
			0
		\end{matrix}
		\right] w(t-1) 
		\end{aligned}
\end{equation}
\begin{equation}
	\begin{aligned}
		y_e^1(t-1) & = \left[
		\begin{matrix}
			C_{11} 	& 0		 	& 0 			& 0 	& 0 	& 0 	& 0 	& 0 	& 0\\
			0		& 0 			& 0			& 0 	& 0 	& 0 	& 0 	& 0 	& 0\\
			0		& 0			& 0			& 0 	& 0 	& 0 	& 0 	& 0 	& 0\\
			0	 	& 0		 	& 0 			& I 	& 0 	& 0 	& 0 	& 0 	& 0\\
			0		& 0	 		& 0			& 0 	& I 	& 0 	& 0 	& 0 	& 0\\
			0		& 0			& 0			& 0 	& 0 	& 0 	& 0 	& 0 	& 0\\		
			0	 	& 0		 	& 0 			& 0	& 0 	& 0 	& I 	& 0 	& 0\\
			0		& 0	 		& 0			& 0 	& 0 	& 0 	& 0 	& I 	& 0\\
			0		& 0			& 0			& 0 	& 0 	& 0 	& 0 	& 0 	& I
			\end{matrix}
		\right]\left[
		\begin{matrix}
			x_1(t-1)\\
			x_2(t-1)\\
			x_3(t-1)\\
			y_1(t-2)\\
			y_2(t-2)\\
			y_3(t-2)\\
			y_1(t-3)\\
			y_2(t-3)\\
			y_3(t-3)
		\end{matrix}
		\right]+\left[
		\begin{matrix}
			D_1\\
			0\\
			0\\
			0\\
			0\\
			0\\
			0\\
			0\\
			0
		\end{matrix}
		\right] w(t-1) 
		\end{aligned}
\end{equation}
The optimal estimate of $x_1(t)$ based on the output $y_e^1(t-1)$ can be obtained from the optimal estimate of $x_e(t)$ based on the output $y_e^1(t-1)$. The computation of the optimal (Kalman) filters is routine and hence omitted here  (consult e. g. \cite{astrom:1970}). 

In a similar way, one can find the
optimal estimates of $x_2(t)$ and $x_3(t)$ based on the corresponding outputs $y_e^2(t-1)$ and $y_e^3(t-1)$. The information available to estimate $\hat{x}_2(t)$ will be $y_2(t-1-s)$ for all $s\geq 0$  (since $[\mathcal{A}^s]_{22}=1, \forall s\geq 0$), and $y_3(t-1-s)$ for all $s\geq 1$  (since $[\mathcal{A}^s]_{23}=1, \forall s\geq 1$). 
$[\mathcal{A}^s]_{21}=0, \forall s\geq 0$, and hence, no measurements of $y_1$ are available.
Finally, the estimate $\hat{x}_3(t)$ will be only based on $y_3(t-s-1)$, $s\geq 0$, since $[\mathcal{A}^s]_{3j}=0$ for $j=1,2$, and $s\geq 0$.
		
Now modify the system matrix $A$ by letting the lower left block matrix $A_{31}\neq 0$. This
implies that we have a cycle of three interconnected systems. Now we get
$$
\mathcal{A}^0= \left[
		\begin{matrix}
			1 	& 0 	& 0\\
			0	& 1 	& 0\\
			0	& 0	& 1
		\end{matrix}
		\right], \hspace{3mm}
\mathcal{A}= \left[
		\begin{matrix}
			1 	& 1 	& 0\\
			0	& 1 	& 1\\
			1	& 0	& 1
		\end{matrix}
		\right], \hspace{3mm} 
	\mathcal{A}^2 = \left[
		\begin{matrix}
			1 	& 2 	& 1\\
			1	& 1 	& 2\\
			2	& 1	& 1
		\end{matrix}
		\right], \hspace{3mm} 
		\mathcal{A}^s = \left[
		\begin{matrix}
			* 	& * 	& *\\
			*	& * 	& *\\
			*	& *	& *
		\end{matrix}
		\right]\hspace{3mm}  \forall s>2.
$$ 
Just as before, the stars stand for positive integers.
Note that the information structure is symmetric (the interconnection graph is symmetric).
Compare with the information structure over a chain. This is a fundamental difference between
cyclic and acyclic graphs. For the cyclic ones, there is a common past (which is 3-steps delayed measurements in the three systems case above), whereas for the acyclic ones, this property is lacking. The property of common past has been used in \cite{rantzer:acc06}.
Nevertheless, the solution structure is the same using our approach, independent of the graphs being cyclic or not. 

In general, we can write the dynamical system in (\ref{distr_estimationP}) as an extended system

\begin{align}
	x_e(t+1) &= A_e x_e(t) + B_e w(t) \label{extended1} \\
	y_e(t)  &= E x_e(t) + D_e w(t) \label{extended2}
\end{align}
where system $i$ measures block component $y_e^i(t)$. The optimal Kalman filter $L_i(\q^{-1})$ in the stationary case
is given by

\begin{align}
	\hat{x}_e(t+1) &= A_e \hat{x}_e(t) + K_i\tilde{y}_e ^i(t)+B_e w(t) \label{estimatei}\\
	\tilde{x}_e(t)    &= x_e(t) - \hat{x}_e(t)\\
	\tilde{y}_e^i(t)  &= E_i \tilde{x}_e(t) + [D_e]_i w(t) \label{innov}
\end{align}

\subsection{Discussion on the Optimal Distributed Controller Structure}
The optimal filter 
$L(\lambda)\in S_{\mathcal{A}}^{n\times m}$
can be written in terms of its rows
$$
L=\left[
  \begin{matrix}
    L_1\\
    \vdots\\
    L_N
  \end{matrix}
  \right],
$$
where $L_i$ is the optimal estimator of the state $x_i(t)$. $L_i$ has the state space realization:
\begin{equation}
\label{Fi2}
\sys {A_e-K_iE_i} {K_i} {\Gamma_i} 0 ,
\end{equation}
with
\begin{equation*}
\Gamma_i=
\left[
  \begin{matrix}
    0 & \cdots & 0 & I & 0 & \cdots & 0
  \end{matrix}
\right],
\end{equation*}
where the identity matrix $I$ in $\Gamma_i$ is in block position $i$, and 
$K_i$ is the optimal Kalman gain.
For instance, comparing with the problem of estimating $x_1(t)$ 
subject to the extended system (\ref{extended}), we have

\begin{equation*}
A_e = \left[
		\begin{matrix}
			A_{11} 	& A_{12} 	& 0 			& 0 	& 0 	& 0 	& 0 	& 0 	& 0\\
			0		& A_{22} 	& A_{23}		& 0 	& 0 	& 0 	& 0 	& 0 	& 0\\
			0		& 0			& A_{33}		& 0 	& 0 	& 0 	& 0 	& 0 	& 0\\
			C_{11} 	& 0		 	& 0 			& 0 	& 0 	& 0 	& 0 	& 0 	& 0\\
			0		& C_{22} 		& 0			& 0 	& 0 	& 0 	& 0 	& 0 	& 0\\
			0		& 0			& C_{33}		& 0 	& 0 	& 0 	& 0 	& 0 	& 0\\
			0	 	& 0		 	& 0 			& I 	& 0 	& 0 	& 0 	& 0 	& 0\\
			0		& 0	 		& 0			& 0 	& I 	& 0 	& 0 	& 0 	& 0\\
			0		& 0			& 0			& 0 	& 0 	& I 	& 0 	& 0 	& 0
		\end{matrix}
		\right] ,
\end{equation*}

$$
E_1 = \left[
		\begin{matrix}
			C_{11} 	& 0		 	& 0 			& 0 	& 0 	& 0 	& 0 	& 0 	& 0\\
			0		& 0 			& 0			& 0 	& 0 	& 0 	& 0 	& 0 	& 0\\
			0		& 0			& 0			& 0 	& 0 	& 0 	& 0 	& 0 	& 0\\
			0	 	& 0		 	& 0 			& I 	& 0 	& 0 	& 0 	& 0 	& 0\\
			0		& 0	 		& 0			& 0 	& I 	& 0 	& 0 	& 0 	& 0\\
			0		& 0			& 0			& 0 	& 0 	& 0 	& 0 	& 0 	& 0\\		
			0	 	& 0		 	& 0 			& 0	& 0 	& 0 	& I 	& 0 	& 0\\
			0		& 0	 		& 0			& 0 	& 0 	& 0 	& 0 	& I 	& 0\\
			0		& 0			& 0			& 0 	& 0 	& 0 	& 0 	& 0 	& I
			\end{matrix}
		\right],
$$
and
\begin{equation*}
\Gamma_1=
\left[
  \begin{matrix}
    I & 0 & & \cdots & 0
  \end{matrix}
\right].
\end{equation*}

For $G = L^T$, we get
$$
G=
\left[
  \begin{matrix}
    L_1^T & L_2^T & \cdots & L_N^T
  \end{matrix}
\right] .
$$
Now let
$$
w=
\left[
  \begin{matrix}
    w_1\\ w_2\\ \vdots \\ w_N
  \end{matrix}
\right] .
$$
Then 
\begin{equation*}
  \begin{aligned}
    u(t) &= -G(\q^{-1})w(t-1)\\
      &= -\sum_{i=1}^N L_i^T(\q^{-1})w_i(t-1).
  \end{aligned}
\end{equation*}
We can see that the controller can be written as the sum of
$N$ controllers, $u(t)=\sum_{i=1}^N u_i(t)$, with $u_i(t)=-F_i^T(\q^{-1})w_i(t-1)$ as the
the feedback law with respect to the disturbance $w_i$ entering system $i$.
Taking the transpose of (\ref{Fi2}) gives the state space
realization of $F_i^T$:
\begin{equation}
\label{FiT}
\sys {A_e^T-E_i^TK_i^T} {\Gamma_i^T} {K_i^T} 0 .
\end{equation}
Let
\begin{equation*}
\Sigma_i:= \begin{array}{ll}
\begin{aligned}
z_i(t+1) =A_e^Tz_i(t)+E_i^Tu_i(t)+\Gamma_iw_i(t).
\end{aligned}
\end{array}
\end{equation*}
It is easy to verify that $u_i(t)=-K_i^Tz_i(t)$ and
$u_i(t)=-F_i^T(\q^{-1})w_i(t-1)$ are equivalent. Hence, the optimal distributed
controller  $u_i(t)=-F_i^T(\q^{-1})w_i(t-1)$ is equivalent to the 
state feedback controller, with respect to the mode generated by
$w_i(t), w_i(t-1), ...$, for $i=1, 2, ..., N$. 
\section{Generalized Distributed Estimation}
Let $W\in \mathbb{S}_{++}^n$, and consider the \emph{weighted} distributed estimation problem
\begin{equation}
\label{gen_estimationP}
    \begin{aligned}
        \inf_{L(\lambda)\in S_{\mathcal{A}}^{n\times m}}\hspace{2mm} & \lim_{M\rightarrow \infty} \frac{1}{M}\sum_{t=1}^M \mathbf{E} (x(t)-\check{x}(t))^T W (x(t)-\check{x}(t)) \\
        \text{subject to }\hspace{2mm}
        & x(t+1) = Ax(t)+Bw(t)\\
        & y(t)=Cx(t)+Dw(t)\\
        & C= \text{diag}(C_{11}, ..., C_{NN})\\
        &\check{x}(t)=\sum_{s=0}^{\infty}l(s) y(t-1-s)\\
        & w(t)=x(t)=x(0)=0\hspace{2mm} \text{for all } t < 0\\
        & w(t)\sim \mathcal{N}(0,I)\hspace{2mm} \text{for all } t\geq 0
    \end{aligned}
\end{equation}
Note that the case $W=I$ reduces to (\ref{estimationP}). The matrix $W$ introduces coupling between
the estimators, so the problem can't be solved through separation as in (\ref{estimationP}). This problem
has been solved for the continuous time case in \cite{barta:sandell}. 
We will give the discrete time analogue following the same proof technique as in \cite{barta:sandell}. It can be seen as an abstraction of the Kalman filter, where the projection theorem of linear algebra
is used sequentially.

First, write the dynamical system in (\ref{gen_estimationP}) on the form (\ref{extended1})-(\ref{extended2}) and  introduce the extended linear dynamical system

\begin{align*}
 X(k+1) &= \mathcal{A}X(k) + \mathcal{B} \mathcal{W}(k)\\
 Y(k) &= \mathcal{C} X(k) + \mathcal{D} \mathcal{W}(k)\\
\end{align*}
where
\begin{align}
\label{extend}
 X(k) &= \mathbf{diag}(x_e(k), ..., x_e(k))\\
  Y(k) &= \mathbf{diag}(y_e^1(k), ..., y_e^N(k))\\
 \mathcal{W}(k) &= \mathbf{diag}(w(k), ..., w(k))\\
    \mathcal{A} &= \mathbf{diag}(A_e, ..., A_e)\\
   \mathcal{B} &= \mathbf{diag}(B_e, ..., B_e)\\
    \mathcal{C} &= \mathbf{diag}(E_1, ..., E_N)\\
   \mathcal{D}  &= \mathbf{diag}([D_e]_1, ..., [D_e]_N)
\end{align}
Then, since $\mathcal{W}(k)$ is white noise, it follows that it is 
$W$-white noise. According to Proposition \ref{equi} in Section \ref{teams}, we can equivalently 
consider the cost
$$
\|X(t) - S(t)\|_W^2
$$
instead of 
$$
\mathbf{E} (x(t)-\check{x}(t))^T W (x(t)-\check{x}(t))
$$
in  the optimization problem (\ref{gen_estimationP}), where $S(t)$ is a causal linear operator 
with column $S_i(t)$ depending only on the output measurements of controller $i$ up to time $t-1$, 
which we will call $\mathcal{Y}_i^{t-1}$. Let $\mathbf{S}^t$ be the space of all causal linear operators $S(t)$ such that $S_i(t)$ depends only $\mathcal{Y}_i^{t-1}$.  
Define $\hat{X}(t)$ and $\tilde{X}(t)$ as

$$
\hat{X}(t) = \arg \min_{S(t)\in\mathbf{S}^t} \|X(t) - S(t)\|_W^2,
$$

$$
\hat{Y}(t) = \mathcal{C}\hat{X}(t),
$$
and
\begin{align*}
\tilde{X}(t) &= X(t) -\hat{X}(t)\\
\tilde{Y}(t) &= Y(t) -\hat{Y}(t)\\ &= \mathcal{C}\tilde{X}(t)+\mathcal{D}\mathcal{W}(t).
\end{align*}
 We have that $\tilde{X}(0) =0$ and $\tilde{X}(0) = X(0)$. $\mathcal{W}(t)$ is orthogonal to the state history
$X(t), X(t-1), ...$, so it's $W$-orthogonal to $\mathbf{S}^t$.
Proposition \ref{orth} gives that $\tilde{X}(t)$ is $W$-orthogonal to $\mathbf{S}^t$, 
so $\tilde{Y}(t)$ is $W$-orthogonal to $\mathbf{S}^t$.
In addition, it follows that $\tilde{Y}(t)$ is
$W$-white noise.
Now introduce
$$
\hat{X}(t+1) = \mathcal{A}\hat{X}(t) + \tilde{S}(k)  .
$$
Then, 
\begin{align*}
\|X(t+1) - \hat{X}(t+1)\|_W^2 
&=\| \mathcal{A}X(t) + \mathcal{B} \mathcal{W}(t) - \hat{X}(t+1)\|_W^2\\
&= \|\mathcal{A}X(t)  - \hat{X}(t+1)\|_W^2  + \|\mathcal{B} \mathcal{W}(t) \|_W^2\\
&= \|\mathcal{A}\hat{X}(t) + \mathcal{A}\tilde{X}(t)  - \hat{X}(t+1)\|_W^2  + \|\mathcal{B} \mathcal{W}(t) \|_W^2\\
&= \|\mathcal{A}\tilde{X}(t) - \tilde{S}(t)\|_W^2 + \|\mathcal{B} \mathcal{W}(t) \|_W^2
\end{align*}
Now combining propositions \ref{equi} and \ref{lin}, we get

$$
K(t) \tilde{Y}(t) = \arg \min_{\tilde{S}(t) \in \mathcal{D}^{t+1}} \|\mathcal{A}\tilde{X}(t) - \tilde{S}(t)\|_W^2,
$$
where
$$
K(t) =  \mathbf{E} \{\tilde{X}(t)W\tilde{Y}^T(t)\}  (\mathbf{E}  \{\tilde{Y}(t)W\tilde{Y}^T(t)\})^{-1}.
$$

Hence, we have that

\begin{align}
\hat{X}(t+1) 	&= \mathcal{A}\hat{X}(t) + K(t)\tilde{Y}(t) \nonumber\\
			&= \mathcal{A}\hat{X}(t) + K(t)(Y(t)-\hat{Y}(t)) \nonumber \\
		   	&= (\mathcal{A}-K(t)\mathcal{C})\hat{X}(t) + K(t)Y(t) \label{teamestimator}
\end{align}
and
\begin{align*}
\tilde{X}(t)   	&= (\mathcal{A}-K(t)\mathcal{C})\tilde{X}(t) - K(t)\mathcal{W}(t) + \mathcal{B}\mathcal{W}(t) 
\end{align*}
Then, the estimator (\ref{teamestimator}) can be written as $N$ separate estimators
with respect to the measurements $y_i(t)$:

\begin{align}
\hat{X}_i(t+1)		&= (\mathcal{A}-K(t)\mathcal{C})\hat{X}_i(t) + K(t)Y_i(t). \nonumber \\
\end{align}
Hence, the estimator $\hat{X}(t)$ can be implemented in a distributed manner. 
Finally, let $\Gamma = I_{n}$ and
\begin{equation*}
\Gamma_j=
\left[
  \begin{matrix}
    0 & \cdots & 0 & I_{n_j} & 0 & \cdots & 0
  \end{matrix}
\right]  ~ ~ ~(\textup{identity matrix in block-position $j$}).
\end{equation*}
We obtain the optimal estimator $\check{x}(t)$ from $\hat{X}(t)$ by using Proposition \ref{equi} with
$L_j = \Gamma_j$. We conclude our result with the theorem below:

\begin{thm}
Consider the weighted distributed estimation problem (\ref{gen_estimationP}). Let
$\hat{X}(0)=0$ and
\begin{align*}
\hat{X}(t+1) 	&= (\mathcal{A}-K(t)\mathcal{C})\hat{X}(t) + K(t)Y(t) \\
\tilde{X}(t+1)   	&= (\mathcal{A}-K(t)\mathcal{C})\tilde{X}(t) - K(t)\mathcal{W}(t) + \mathcal{B}\mathcal{W}(t) 
\end{align*}
with
$$
K(t) =  \mathbf{E} \{\tilde{X}(t)W\tilde{Y}^T(t)\}  (\mathbf{E}  \{\tilde{Y}(t)W\tilde{Y}^T(t)\})^{-1}.
$$
Partition $\hat{X}$ in $N$ blocks of $n\times n$ matrices $[\hat{X}]_{ji}$.
Then, the optimal estimator is given by 
$$
\check{x}_i(t) = \sum_{j=1}^N = \Gamma_j \hat{X}_{ji}(t).
$$

\end{thm}
%
\section{Distributed State Feedback with Cross-Correlation in the Disturbance}
The distributed state feedback given by (\ref{controlP}) considers the case where 
$w(t)\sim \mathcal{N}(0,I)$, that is $w_i(t)$ is uncorrelated with $w_j(t)$ for $i\neq j$.
We will now consider a slightly different problem where $w(t)\sim \mathcal{N}(0,W)$
for a general positive definite matrix $W$:

\begin{equation}
\label{controlPW}
    \begin{aligned}
        \inf_{K(\lambda)\in S_{\mathcal{A}}^{m\times n}}\hspace{2mm} & \lim_{M\rightarrow \infty} \frac{1}{M}\sum_{t=1}^M \sum_{i=1}^N \mathbf{E} \|z_i(t)\|^2 \\
        \text{subject to }\hspace{2mm} & x(t+1) = Ax(t)+Bu(t)+w(t)\\
        & z(t)=Cx(t)+Du(t)\\
        & B= \mathbf{diag}(B_{11}, ..., B_{NN})\\
        & u(t)=\sum_{s=0}^{\infty} k(s) x(t-s)\\
        & w(t)=x(t)=x(0)=0\hspace{2mm} \text{for all } t < 0\\
        & w(t)\sim \mathcal{N}(0,W)\hspace{2mm} \text{for all } t\geq 0 
    \end{aligned}
\end{equation}

Following the same arguments as the proof of Theorem \ref{equivalencethm}, we see that
(\ref{controlPW} is equivalent to the feedforward control problem 
\begin{equation}
\label{feedforwardW}
    \begin{aligned}
        \inf_{G(\lambda)\in S_\mathcal{A}^{m\times n}}\hspace{2mm} & \lim_{M\rightarrow \infty} \frac{1}{M}\sum_{t=1}^M \sum_{i=1}^N \mathbf{E} \|z_i(t)\|^2 \\
        \text{subject to }\hspace{2mm} & x(t+1) = Ax(t)+Bu(t)+w(t)\\
        & B= \mathbf{diag}(B_{11}, ..., B_{NN})\\
        & z(t)=Cx(t)+Du(t)\\
        & u(t)=-\sum_{s=0}^{\infty} g(s)w(t-1-s)\\
        & w(t)=x(t)=x(0)=0\hspace{2mm} \text{for all } t < 0\\
        & w(t)\sim \mathcal{N}(0,W)\hspace{2mm} \text{for all } t\geq 0
   \end{aligned}
\end{equation}
where the only change is in that $w(t)\sim \mathcal{N}(0,W)$. It is also straightforward to apply the proof of
Theorem \ref{estimationthm} to show that the dual of (\ref{feedforwardW}) is given by
the weighted distributed estimation problem (\ref{gen_estimationP}), since the dynamics can be written as
$$
x(t+1) = Ax(t)+Bu(t)+W^{\frac{1}{2}}w(t)
$$
where $w(t)\sim \mathcal{N}(0,I)$ is white noise.
We have already seen that 
problem (\ref{gen_estimationP}) is conceptually more general for general weight matrices $W$. Therefore, 
correlation in the disturbance for distributed state feedback control changes the problem substantially.

%


\section{Conclusion}
We showed that distributed estimation and control problems are
dual under partially nested information pattern using a novel
system theoretic formulation of dynamics over graphs. We showed that the distributed estimation problem can be decomposed into $N$ separate problems that are easy to solve, and hence solve the corresponding distributed control problem due to the duality that was shown in this paper.  We considered a distributed estimation problem formulated as a dynamical team problem. We proposed a solution based on linear quadratic team decision theory, which provides a generalized Riccati equation for teams. We also showed that the weighted estimation problem is the dual to a distributed state feedback problem, where the disturbances entering the systems are correlated, and hence, a solution is obtained based on generalized Riccati equation for teams.



\bibliography{../../ref/mybib}

\appendix

\subsection{Graph Theory} \label{graphtheory}
A (simple) graph $\mathcal{G}$ is an
ordered pair $\mathcal{G}: = (\mathcal{V},\mathcal{E})$ where
$\mathcal{V}$ is a set, whose elements are called
\textit{vertices} or \textit{nodes}, $\mathcal{E}$ is a set of
pairs (unordered) of distinct vertices, called \textit{edges} or
\textit{lines}. The set $\mathcal{V}$ (and hence $\mathcal{E}$) is
taken to be finite in this paper. A \textit{loop} is an edge which
starts and ends with the same node.

A directed graph or digraph $\mathcal{G}$ is a graph where
$\mathcal{E}$ is a set of ordered pairs of vertices, called
\textit{directed} edges, arcs, or arrows.
An edge $e = (v_i,v_j)$ is considered to be directed from $v_i$ to $v_j$;
$v_j$ is called the head and $v_i$ is called the tail of the edge.

The \textit{adjacency matrix} of a finite directed graph
$\mathcal{G}$ on $n$ vertices is the $n\times n$ matrix where the
nondiagonal entry $\mathcal{A}_{ij}$ is the number of edges from
vertex $j$ to vertex $i$, and the diagonal entry
$\mathcal{A}_{ii}$ is the number of loops at vertex $i$ (the
number of loops at every node is defined to be one, unless another
number is given on the graph).

\end{document}